\documentclass[11pt]{article}

\usepackage{amsfonts,amsmath,amsthm,amscd,amssymb,latexsym,cite,verbatim,texdraw,floatflt,
caption2,pb-diagram,url}
\theoremstyle{plain}
\newtheorem{theorem}{Theorem}[section]

\newtheorem{corollary}{Corollary}[section]
\newtheorem{definition}{Definition}[section]
\theoremstyle{definition}
\newtheorem{example}{Example}[section]

\newcommand{\keywords}{\textbf{Key words. }\medskip}
\newcommand{\subjclass}{\textbf{MSC 2020. }\medskip}
\renewenvironment{abstract}{\par\noindent\textbf{Abstract. }\medskip}{\par}

\numberwithin{equation}{section}

\begin{document}

\title{Best Proximity Point Results for Perimetric Contractions}

\author{Hiranmoy Garai, Evgeniy Petrov,\\
Pratikshan Mondal and Lakshmi Kanta Dey}

\maketitle

\begin{abstract}
This paper has two aims, first one is  to introduce special kind of  proximal contractions guaranteeing a finite number of best proximity points, and second one is to  derive best proximity point results for perimetric contractions. To meet these two aims, we introduce  two new proximal contractions: perimetric proximal contractions of the first and the second kind, and derive best proximity point results for these mappings. We establish that for these particular mappings, best proximity points are not necessarily unique; however, we provide an upper bound, proving that at most two such points can exist.  To establish the validity of our results, we provide illustrative examples demonstrating that these newly defined mappings can possess unique or exactly two best proximity points.
\end{abstract}

\subjclass{47H10, 54H25.}

\keywords{Best proximity point, perimetric proximal contractions, approximative compactness}

\section{Introduction}
Fixed point theory is an invaluable tool to find solutions of equations of the form $Tx=x$ ($T$ being a self-map on a non empty set $A$, say). However, if $T$ isn't a self-map, specifically if  the domain and co-domain of $T$ are disjoint, fixed point theory isn't applicable for solving such equations. In fact, exact solutions to these equations are unavailable, so we aim for an optimal solution instead. If $A$ and $B$ are disjoint non-empty subsets of a metric space $(X,d)$, and $T:A\to B$ is a mapping, then finding an optimal solution to $Tx=x$ means finding $x^*\in A$ such that $d(x^*,Tx^*)=d(A,B)$, where $d(A,B)=\inf \{  d(a,b):a\in A, b\in B\}$. Such a point $x^* \in A$, if it exists, satisfying $d(x^*,Tx^*)=d(A,B)$, is referred to as a best proximity point of $T$. The concept of best proximity points was initially considered by Fan \cite{F} and Reich \cite{R16,R17}, and subsequently explored extensively by mathematicians in various directions throughout the last few years, see \cite{AV,KRR,B8,B11,K8,MD}. Among these various directions, most involve different types of contraction conditions, some of which have been used previously in fixed point theory. The various contraction conditions specifically developed for best proximity point results include many interesting contractions. Notably, the proximal contractions introduced by Sadiq Basha \cite{B10} are among the most significant ones. 

\begin{definition}\cite[p.~1774, Definition~2.2]{B10}\label{d01}
Let $A$ and $B$ be disjoint non-empty subsets of a metric space $(X,d)$. A mapping $T: A\to B$ is said to  be a \emph{proximal contraction of the first kind} if there exists a constant $\alpha\in [0,1)$ such that for all $u_1,u_2,x_1,x_2\in A$, the following hold:
$$
\left.\begin{aligned}
d(u_1,Tx_1)= d(A,B)\\
d(u_2,Tx_2)= d(A,B) 
\end{aligned}\right\rbrace \implies \begin{aligned}d(u_1,u_2)\leq  \alpha d(x_1,x_2). \end{aligned}$$
\end{definition}

\begin{definition}\cite[p.~1774, Definition~2.4]{B10}
Let $A$ and $B$ be disjoint non-empty subsets of a metric space $(X,d)$. A mapping $T: A\to B$ is said to  be a \emph{proximal contraction of the second kind} if there exists a constant $\alpha\in [0,1)$ such that for all $u_1,u_2,x_1,x_2\in A$, the following hold:
$$
\left.\begin{aligned}
d(u_1,Tx_1)= d(A,B)\\
d(u_2,Tx_2)= d(A,B) 
\end{aligned}\right\rbrace \implies \begin{aligned}d(Tu_1,Tu_2)\leq  \alpha d(Tx_1,Tx_2). \end{aligned}$$
\end{definition}

The key best proximity point findings concerning these proximal contractions as follows, with definitions of $A_0$ and $B_0$ recalled beforehand: 

For a metric space $(X,d)$ and two non-empty subsets $A,B$ of $X$, we use the notations $A_0$ and $B_0$ as follows:
\begin{align*}
&A_0=\{x\in A:d(x,y)=d(A,B)~~\mbox{for some}~~y\in B\}\\
&B_0=\{y\in B:d(x,y)=d(A,B)~~\mbox{for some}~~x\in A\}.
\end{align*}

\begin{theorem}\cite[p.~1778, Corollary~3.4]{B10}
If $(X,d)$ is complete; $A,B$ are closed, $A_0, B_0$ are non-empty; $T(A_0)\subseteq B_0$ and $T$ is a continuous proximal contraction of the first kind, then $T$ has a unique best proximity point in $A$.
\end{theorem}

\begin{theorem}\cite[p.~1777, Corollary~3.2]{B10}
If $(X,d)$ is complete; $A,B$ are closed, $A_0, B_0$ are non-empty; $A$ is approximatively compact with respect to $B$; $T(A_0)\subseteq B_0$ and $T$ is a continuous proximal contraction of the second kind, then $T$ has a best proximity point in $A$. Moreover, if $x,y\in A$ are best proximity points of $T$, then $Tx=Ty$.
\end{theorem}

Following Basha's pioneering work, proximal contractions have been explored in various ways (see \cite{B8, CXWD, MGD}). It's noteworthy that most best proximity point results conclude with the existence of either a unique or infinitely many best proximity points.

On the other hand, a recent notable development in fixed point theory is the concept of mappings that contract perimeters of triangles, introduced by Petrov \cite{P9} in  $2023$.

\begin{definition}\cite[p.~2, Definition~2.1]{P9}
Let $(X,d)$ be a metric space with at least three points. Then  $T: X \to X$ is said to be a mapping contracting perimeters of triangles on $X$ if there exists $\alpha \in [0,1)$ such that the inequality
    \begin{equation}\label{mcpt1}
        d(Tx,Ty)+ d(Ty,Tz)+ d(Tz,Tx) \le \alpha (d(x,y)+ d(y,z)+ d(z,x))
        \end{equation}
        holds for all three pairwise distinct points $x, y, z \in X$.
\end{definition}

This pioneering approach is complemented by a foundational fixed point existence result, which goes as follows:
\begin{theorem}\cite[p.~3, Theorem~2.4]{P9}\label{T1}
Let $(X,d)$, $|X| \geq 3$, be a complete metric space and let $T:X\to X$ be a mapping contracting perimeters of triangle on $X$. Then $T$ has a fixed point in $X$ if and only if $T$ does not possess any periodic points of prime period $2$. The number of fixed points is at most two.
\end{theorem}
This area has been further explored by  mathematicians in various ways (see \cite{B12, PSB, PB23,BPS25, BMD25, BMD26, BMDS26}), though each effort is limited to fixed points only. A striking aspect of these new mappings is that they often yield a finite number of fixed points, unlike most existing results which focus on unique or infinite fixed points. 

These facts raise two questions. First, why are perimeter-contracting mappings restricted to fixed points, and not extended to best proximity points? The second question is: can we have best proximity point results with a finite number of best proximity points? With these questions in mind, we explore the perimetric counterparts of Sadiq Basha's proximal contractions, seeking answers to both. More precisely, we propose two kinds of perimetric proximal contractions:  perimetric proximal contraction of the first kind and  perimetric proximal contraction of  the second kind. After this, we obtain two results guaranteeing a finite number of best proximity points, which aren't necessarily unique- in fact, we show the number of best proximity points is at most two.  These results are accompanied by examples to illustrate their validity.

\section{Perimetric proximal contractions}\label{sec2}

Here, we  define perimetric proximal contractions of the first and  the second kind formally.
\begin{definition}\label{d1}
Let $A,B$ be two non-empty subsets of a metric space $(X,d)$, and let $T:A\to B$ be a mapping. Then $T$ is called a \emph{perimetric proximal contraction of  the first kind} if there exists $\alpha\in [0,1)$ such that for all $u_1,u_2,u_3,x_1,x_2,x_3\in A$ with  $x_1,x_2,x_3$ are pairwise distinct, the following hold:
$$
\left.\begin{aligned}
d(u_1,Tx_1)= d(A,B)\\
d(u_2,Tx_2)= d(A,B) \\
d(u_3,Tx_3)= d(A,B)
\end{aligned}\right\rbrace \implies \begin{aligned}d(u_1,u_2)+d(u_2,u_3)+d(u_3,u_1)\leq \\ \alpha\big\{ d(x_1,x_2)+d(x_2,x_3)+d(x_3,x_1)  \big\} \end{aligned}$$
\end{definition}

\begin{definition}\label{d2}
Let $A,B$ be two non-empty subsets of a metric space $(X,d)$, and let $T:A\to B$ be a mapping. Then $T$ is called a \emph{perimetric proximal contraction of  the second kind} if there exists $\alpha\in [0,1)$ such that for all $u_1,u_2,u_3,x_1,x_2,x_3\in A$ with  $Tx_1,Tx_2,Tx_3$ are pairwise distinct, the following hold:
$$
\left.\begin{aligned}
d(u_1,Tx_1)= d(A,B)\\
d(u_2,Tx_2)= d(A,B) \\
d(u_3,Tx_3)= d(A,B)
\end{aligned}\right\rbrace \implies \begin{aligned}d(Tu_1,Tu_2)+d(Tu_2,Tu_3)+d(Tu_3,Tu_1)\leq \\ \alpha\big\{ d(Tx_1,Tx_2)+d(Tx_2,Tx_3)+d(Tx_3,Tx_1)  \big\} \end{aligned}$$
\end{definition}

Let $A$ and $B$ be disjoint non-empty subsets of a metric space $(X, d)$, $|A|\geqslant 3$, and let $T\colon  A \to B$ be a mapping. It is easy to see that if for any $u_1, u_2, u_3, x_1, x_2, x_3 \in  A$ with $x_1, x_2, x_3$ are pairwise distinct for which the following hold:
$d(u_1, Tx_1) = d(A,B)$,
$d(u_2, Tx_2) = d(A,B)$,
$d(u_3, Tx_3) = d(A,B)$, then any proximal contraction of the first kind is a perimetric proximal contraction of the first
kind. Indeed, by Definition~\ref{d01} the following inequalities hold:
$$
d(u_1,u_2)\leqslant \alpha d(x_1,x_2), \quad
d(u_2,u_3)\leqslant \alpha d(x_2,x_3), \quad
d(u_3,u_1)\leqslant \alpha d(x_3,x_1).
$$
Summarizing the left and the right parts of these inequalities we get exactly the respective inequality in Definition~\ref{d1}.

In the following example we show that a perimetric proximal contraction of the first kind is not obligatory proximal contraction of the first kind.

\begin{example}\label{ex23}
Consider $\mathbb R$ with the Euclidean metric, and set
$$
A=\{-3,0,3,4\},\qquad B=[-2,-1]\cup[1,2].
$$
Define $T:A\to B$ by
$$
T(-3)=-2,\qquad T(0)=\frac32,\qquad T(3)=2,\qquad T(4)=1.
$$
We have $d(A,B)=1$, and the proximal pairs are precisely
$$
\{(u,x)\in A\times A:d(u,Tx)=1\}
=\{(-3,-3),(3,3),(0,4)\}.
$$
Consequently, whenever $x_1,x_2,x_3$ are pairwise distinct and
$d(u_i,Tx_i)=1$ for $i=1,2,3$, the triples $(x_1,x_2,x_3)$
and $(u_1,u_2,u_3)$ are simultaneous permutations of
$(-3,3,4)$ and $(-3,3,0)$, respectively. Their perimeters satisfy
$$
d(u_1,u_2)+d(u_2,u_3)+d(u_3,u_1)=12
$$
and
$$
d(x_1,x_2)+d(x_2,x_3)+d(x_3,x_1)=14.
$$
Thus, $T$ is a perimetric proximal contraction of the first kind
with $\alpha=6/7$.

Moreover,
$$
A_0=\{-3,0,3\},\qquad B_0=\{-2,-1,1,2\},
$$
whereas
$$
T(A_0)=\{-2,3/2,2\}\nsubseteq B_0.
$$
The best proximity points of $T$ are exactly $-3$ and $3$.
Since
$$
d(-3,T(-3))=d(3,T(3))=1,
$$
a proximal contraction of the first kind would have to satisfy
$$
6=d(-3,3)\leq\alpha d(-3,3)=6\alpha
$$
for some $\alpha<1$, which is impossible.
\end{example}

Moving on, we present our main results on best proximity points. First, we formulate a special condition for a mapping $T$, which will be crucial for our findings. We designate this special condition as Condition~$\Lambda$.

\textbf{Condition~$\Lambda$:} Let $A,B$ be two non-empty subsets of a metric space $(X,d)$ and let $T:A\to B$ be a mapping. Then, $T$ is said to satisfy Condition $\Lambda$ if there does not exist any pair of points $(x,y)$  where $x\neq y$, of $A$ such that $d(x,Ty)=d(y,Tx)=d(A,B)$.

We now present our first result.

\begin{theorem}\label{t1}
Suppose that $(X,d)$ is a complete metric space and $A,B$ are closed subsets of $X$. Let $|A_0| \geq 3$. Let $T:A\to B$ be continuous perimetric proximal contraction of  the first kind such that $T(A_0)\subseteq B_0$. Then if $T$ satisfies  Condition $\Lambda$,  $T$ possesses a best proximity point. Moreover, the best proximity points of $T$ are at most two in number.
\end{theorem}
\begin{proof}
 Since $|A_0| \geq 3$, we can choose an element $u_0\in A_0$. Then, $Tu_0\in T(A_0)\subseteq B_0$, so there exists $u_1\in A_0$ such that $d(u_1,Tu_0)=d(A,B)$. Again, $Tu_1\in T(A_0)\subseteq B_0$, so there exists $u_2\in A_0$ such that $d(u_2,Tu_1)=d(A,B)$. Continuing in this way, we obtain a sequence $\{u_n\}$ having the properties that $u_n\in A_0$ and $d(u_{n+1},Tu_n)=d(A,B)$ for all $n\in \mathbb{N}$.

If $u_n = u_{n+1}$ for some $n \in \mathbb{N}$, it follows that $u_n$ is a best proximity point of $T$, as $d(u_{n+1}, Tu_n) = d(u_n, Tu_n) = d(A,B)$.

Next, we assume that $u_n \neq u_{n+1}$ for all $n \in \mathbb{N}$. Under this assumption, we claim that $u_n,u_{n+1},u_{n+2}$ are pairwise distinct for all $n \in \mathbb{N}$. If not, then we must have $u_n=u_{n+2}$ for some $n \in \mathbb{N}$. Then $d(u_n,Tu_{n+1})=d(u_{n+2}, Tu_{n+1})=d(A,B)$ and also $d(u_{n+1}, Tu_{n})=d(A,B)$. Hence, $T$ does not satisfy Condition $\Lambda$,  contradicting the assumption made earlier. Thus, any three consecutive terms of $\{u_n\}$ are pairwise distinct. 

Now, since 
$$
d(u_{n+1},Tu_{n})=d(A,B)=d(u_{n+2},Tu_{n+1})=d(u_{n+3},Tu_{n+2}),
$$ 
we have 
\begin{align}
d(u_{n+1},u_{n+2})+&d(u_{n+2},u_{n+3})+d(u_{n+3},u_{n+1})\leq \nonumber\\
& \alpha\{d(u_{n},u_{n+1})+d(u_{n+1},u_{n+2})+d(u_{n+2},u_{n})\}. \label{e1}
\end{align}
Let $t_n=d(u_{n},u_{n+1})+d(u_{n+1},u_{n+2})+d(u_{n+2},u_{n})$. Then from \eqref{e1}, we get 
\begin{equation}\label{e2}
t_{n+1}\leq \alpha t_n~~~\mbox{for all}~~~n\in \mathbb{N}.
\end{equation}
Iterating \eqref{e2}, we obtain
$$
t_j\leq \alpha^{j-1}t_1\qquad (j\geq2).
$$
Consequently, for all integers $m>n\geq2$,
$$
d(u_n,u_m)
\leq \sum_{j=n}^{m-1}d(u_j,u_{j+1})
\leq \sum_{j=n}^{m-1}t_j
\leq t_1\sum_{j=n}^{m-1}\alpha^{j-1}
\leq \frac{\alpha^{n-1}t_1}{1-\alpha}.
$$
Since $0\leq\alpha<1$, the right-hand side tends to zero
as $n\to\infty$. Thus, $\{u_n\}$ is a Cauchy sequence.

Since $A$ is a closed subset of the complete metric space
$(X,d)$, there exists $u\in A$ such that $u_n\to u$
as $n\to\infty$.
Given that $T$ is continuous, we have $Tu_n \to Tu$, and consequently, $d(u_{n+1}, Tu_n) \to d(u,Tu)$ as $n \to \infty$. Thus, we have $d(u,Tu)=d(A,B)$, implying that $u$ is a best proximity point of $T$. 

For the last part of the proof, we assume that $T$ has three distinct (which are pairwise distinct as well) best proximity points $x,y,z\in A$. Then, we have $d(x,Tx)=d(A,B)=d(y,Ty)=d(z,Tz)$ and so, 
\begin{align*}
d(x,y)+d(y,z)+d(z,x)&\leq \alpha\{d(x,y)+d(y,z)+d(z,x)\}\\
&< d(x,y)+d(y,z)+d(z,x),
\end{align*}
leading to a contradiction, and thus the proof is concluded.
\end{proof}
From Theorem \ref{t1}, we have the following corollary:
\begin{corollary}\label{cor1}
 The sufficiency holds in Theorem \ref{T1}.   
\end{corollary}

\begin{proof}
  Indeed, let us set $A=B=X$ in Theorem \ref{t1}. Then it follows from Definition \ref{d1} that $d(A,B)=0$, $u_1=Tx_1$, $u_2=Tx_2$, $u_3=Tx_3$ and, hence, $T$ is a mapping contracting the perimeters of triangles. The completeness of the space $X$ ensures that every Cauchy sequence converges to a limit that belongs to the set $X$. In the same way, the closedness of $A$ and $B$ provided similar property for these subsets. Clearly, $A_0=B_0=X$ and the inclusion $T(A_0)\subseteq B_0$ holds. One can easily see that in the case $A=B=X$, Condition $\Lambda$ is equivalent to the condition that there is no periodic point of prime period $2$ and the existence of a best proximity point is equivalent to the existence of a fixed point for the mapping $T$. It remains only to note that any mapping contracting the perimeters of triangles is continuous (\cite[\, p. 2, Proposition 2.3]{P9}).
\end{proof}
Our next result is for perimetric proximal contraction of the second kind. Before this, we recall the definition of approximatively compact set. 

\begin{definition}\cite[p.~1774, Definition~2.1]{B10}
Let $A,B$ be two non-empty subsets of a metric space $(X,d)$. Then
$A$ is said to be approximatively compact with respect to $B$ if every sequence $\{u_n\}$ of $A$ satisfying the condition $d(u,u_n)\to d(u,A)$ for some $u\in B$, has a convergent
subsequence.
\end{definition}
\begin{theorem}\label{t2}
The  conclusions of Theorem \ref{t1} remain valid if $T$ is a perimetric proximal contraction of the second kind instead of  the first kind  with other conditions of Theorem \ref{t1}, and $A$ is approximatively compact with respect to $B$. 
\end{theorem}

\begin{proof}
Following the proof of Theorem \ref{t1}, we construct a sequence $\{u_n\}$ in $A$ such that $u_n\in A_0$, $d(u_{n+1},Tu_n)=d(A,B)$ for all $n\in \mathbb{N}$.  If $u_n=u_{n+1}$ for some $n$, then $u_n$ is a best proximity point of $T$. So, we assume that $u_n \neq u_{n+1}$ for all $n \in \mathbb{N}$, and  we can then show, as before, that any three consecutive terms of $\{u_n\}$ are pairwise distinct. If $Tu_n=Tu_{n+1}$ for some $n$, then we have $d(u_{n+1},Tu_{n+1})=d(u_{n+1},Tu_{n})=d(A,B)$, showing that $u_{n+1}$ is a best proximity point of $T$. We therefore assume that $Tu_n \neq Tu_{n+1}$ for every $n$. Again, if $Tu_n=Tu_{n+2}$ for some $n$, then we have $d(u_{n+1},Tu_{n+2})=d(u_{n+1},Tu_{n})=d(A,B)$ and also, we have $d(u_{n+2},Tu_{n+1})=d(A,B)$, showing that $T$ does not satisfy Condition $\Lambda$ which leads to a contradiction. Hence, any three consecutive terms of $\{Tu_n\}$ are pairwise distinct.

Following the lines of proof of Theorem \ref{t1},  we can show that $\{Tu_n\}$ is a Cauchy sequence in the closed subset $B$  of $X$, and hence there exists $v\in B$ such that $Tu_n\to v$ as $n\to \infty$. Then, we have
\begin{align*}
d(v,A)\leq d(v,u_{n+1}) &\leq d(v,Tu_n)+d(Tu_n,u_{n+1})\\
&=d(v,Tu_n)+d(A,B)\\
&\leq d(v,Tu_n)+d(v,A).
\end{align*}
Using the fact that $\lim_{n\to \infty} d(v,Tu_n)=0$, we obtain  $\lim_{n\to \infty} d(v,u_{n+1})=d(v,A).$ As $A$ is approximatively compact with respect to $B$, $\{u_n\}$ has a convergent subsequence $\{u_{n(k)}\}$ converging to $u\in A$. Then
\begin{equation}\label{e4}
d(u,v)=\lim_{k\to\infty}d(u_{n(k)},Tu_{n(k)-1})=d(A,B),
\end{equation}
 and so $u\in A_0$. Further, since $T$ is continuous, we have $\lim_{k\to\infty}Tu_{n(k)}=Tu$, and so $Tu=v$. Thus from \eqref{e4}, we get $d(u,Tu)=d(A,B)$, showing that $u$ is a best proximity point of $T$.
 
We show that $T$ has at most two best proximity points. Assume, if possible, that $T$ has three distinct best proximity points $x, y, z \in A$. If two of these images coincided, say $Tx=Ty$, then
$$
d(x,Ty)=d(x,Tx)=d(A,B)
\,\text{ and }\,
d(y,Tx)=d(y,Ty)=d(A,B).
$$
Since $x\neq y$, these equalities contradict
Condition~$\Lambda$. Therefore, $Tx$, $Ty$ and $Tz$
are pairwise distinct. Then from Definition \ref{d2}, we get
 \begin{align*}
&d(Tx,Ty)+d(Ty,Tz)+d(Tz,Tx) \\ &\leq \alpha\{d(Tx,Ty)+d(Ty,Tz)+d(Tz,Tx)\}\\
&< d(Tx,Ty)+d(Ty,Tz)+d(Tz,Tx),
\end{align*}
resulting in a contradiction. Hence, $T$ has at most two best proximity points.
\end{proof}

\begin{corollary}
Let $(X, d)$, $|X| \geqslant 3$, be a complete metric
space and let $T\colon X \to X$ be a continuous mapping such that the inequality
\begin{multline*}
  d(T^2x, T^2y) + d(T^2y, T^2z) + d(T^2z, T^2x) \\
  \leqslant \alpha \{d(Tx, Ty) + d(Ty, Tz) + d(Tz, Tx)\}
\end{multline*}
holds for all  $x,y,z \in X$ with $Tx,Ty,Tz$ being pairwise distinct and some $0\leqslant \alpha<1$. Suppose also that $T$ does not possess any periodic points of prime period $2$. Then
$T$ has a fixed point in $X$. The number of fixed points is at most two.
\end{corollary}
\begin{proof}
The proof is similar to the proof of Corollary~\ref{cor1}. 
\end{proof}
We now present a series of carefully constructed examples that illustrate and substantiate the theoretical results obtained.
\begin{example}\label{ex1}
Consider a metric space $(\mathbb R^2,d)$, where $d$ is the Euclidean metric on $\mathbb R^2$.
Let $A=\{p,q,r,s\}$, $B=\{a,b,c\}$,
where
$a=(0,0)$,
$b=(0,1)$,
$c=(0,2)$,
$p=(1,2)$,
$q=(1,1)$,
$r=(1,0)$,
$s=(2,0)$, and let
$$
T(p)=c, \quad T(q)= b, \quad T(r)=b, \quad T(s)=a.
$$ 

It is clear that $d(A,B)=1=d(a,r)=d(b,q)=d(c,p)$, $A_0=\{p,q,r\}$, $B_0=\{a,b,c\}$.

Let us show that $T$ is a perimetric proximal contraction of the first kind. Let us identify all pairs $(u, x) \in A \times A$ that satisfy 
the condition $d(u, Tx) = d(A, B) = 1$.

Let us test each $x \in A$ to find corresponding $u \in A$:
\begin{itemize}
    \item If $x = p$, $Tx = c$. We need $d(u, c) = 1$. The only solution in $A$ is $u = p$, which gives the pair $(p,p)$.
    \item If $x = q$, $Tx = b$. We need $d(u, b) = 1$. The only solution in $A$ is $u = q $, which gives the pair $(q,q)$.  
    \item If $x = r$, $Tx = b$. We need $d(u, b) = 1$. The only solution in $A$ is $u = q$, which gives the pair $(q, r)$.
    \item If $x = s$, $Tx = a$. We need $d(u, a) = 1$. The only solution in $A$ is $u = r$, which gives the pair $(r, s)$.
\end{itemize}

The definition requires the inequality to hold for $x_1, x_2, x_3$ that are pairwise distinct. There are only four possible pairs of sets of points $\{u_1, u_2, u_3\}$ and $\{x_1, x_2, x_3\}$ (up to permutation) that satisfy the necessary condition. Specifically:
$$
 \{u_1, u_2, u_3\} = \{p, q, r\} \quad \text{with} \quad \{x_1, x_2, x_3\} = \{p, q, s\},
$$
$$
 \{u_1, u_2, u_3\} = \{q, q, r\} \quad \text{with} \quad \{x_1, x_2, x_3\} = \{q, r, s\},
$$
$$
 \{u_1, u_2, u_3\} = \{p, q, q\} \quad \text{with} \quad \{x_1, x_2, x_3\} = \{p, q, r\},
$$
and
$$
\{u_1, u_2, u_3\} = \{p, q, r\} \quad \text{with} \quad \{x_1, x_2, x_3\} = \{p, r, s\}.
$$
For all $x,y,z\in X$ set $P(x,y,z)=d(x,y)+d(y,z)+d(z,x)$.

Case I:
$$
P(p,q,r) = 1+1+2=4, \quad
P(p,q,s) =1+\sqrt{2}+\sqrt{5}.
$$

Case II:
$$
P(q,q,r) = 0+1+1=2, \quad
P(q,r,s) =1+1+\sqrt{2}=2+\sqrt{2}.
$$

Case III:
$$
P(p,q,q) = 1+0+1=2, \quad
P(p,q,r) =1+1+2=4.
$$

Case IV:
$$
P(p,q,r)=1+1+2=4,\quad
P(p,r,s) =2+1+\sqrt{5}=3+\sqrt{5}.
$$
Clearly, in each case the inequality in Definition \ref{d1}
 holds with some $\alpha<1$.

Thus, $T$ is a perimetric proximal contraction of the first kind.
Additionally, from the form of $T$, it is evident that there are no distinct points $x, y \in A$ for which the equalities $d(x, Ty) = d(A,B) = d(y, Tx)$ hold, which ensures that $T$ satisfies Condition $\Lambda$. 
Also, it is easy to see that $T(A_0)=\{b,c\}\subseteq B_0$.
Accordingly, all the assumptions of Theorem \ref{t1} hold, so by virtue of Theorem \ref{t1}, $T$ has best proximity points, limited to a maximum of two. 
We remark that $T$ has two best proximity points $p$ and $q$.
\end{example}

\begin{example}\label{ex210}
Consider the Hilbert space $\ell^2$ with
$d(x,y)=\|x-y\|_2$, and let $\{e_n\}_{n=1}^{\infty}$ be its
standard orthonormal basis. Set
$$
A=\{e_1,e_2,e_3,e_4\},\qquad
B=\{f_1,f_2,f_3\},\qquad f_j=2e_j\quad(j=1,2,3).
$$
Define $T:A\to B$ by
$$
T(e_1)=f_1,\qquad T(e_2)=T(e_3)=T(e_4)=f_2.
$$
The space $\ell^2$ is complete, the sets $A$ and $B$ are
finite and hence closed, and $T$ is continuous because $A$
is discrete. For $1\leq i\leq4$ and $1\leq j\leq3$,
$$
d(e_i,f_j)=
\begin{cases}
1,&i=j,\\
\sqrt5,&i\neq j.
\end{cases}
$$
It follows that
$$
d(A,B)=1,\qquad A_0=\{e_1,e_2,e_3\},\qquad B_0=B,
$$
so $|A_0|=3$ and $T(A_0)=\{f_1,f_2\}\subseteq B_0$.

The proximal pairs are exactly
$$
\{(u,x)\in A\times A:d(u,Tx)=1\}
=\{(e_1,e_1),(e_2,e_2),(e_2,e_3),(e_2,e_4)\}.
$$
For any pairwise distinct $x_1,x_2,x_3\in A$, their perimeter
is $3\sqrt2$. If $d(u_i,Tx_i)=1$ for $i=1,2,3$, each $u_i$
belongs to $\{e_1,e_2\}$, and consequently
$$
d(u_1,u_2)+d(u_2,u_3)+d(u_3,u_1)
\leq2\sqrt2
=\frac23\bigl(d(x_1,x_2)+d(x_2,x_3)+d(x_3,x_1)\bigr).
$$
Thus, $T$ is a perimetric proximal contraction of the first
kind with $\alpha=2/3$.

If distinct $x,y\in A$ satisfied
$d(x,Ty)=d(y,Tx)=1$, the list of proximal pairs would force
$\{x,y\}=\{e_1,e_2\}$. However,
$d(e_1,T(e_2))=\sqrt5$, so such a pair cannot exist.
Therefore $T$ satisfies Condition~$\Lambda$.

All the hypotheses of Theorem~\ref{t1} are satisfied.
In fact,
$$
d(e_i,T(e_i))=
\begin{cases}
1,&i=1,2,\\
\sqrt5,&i=3,4.
\end{cases}
$$
Hence $e_1$ and $e_2$ are exactly the best proximity points,
showing that the upper bound in Theorem~\ref{t1} is attained.
\end{example}

\begin{example}\label{ex2}
Consider $\mathbb R$ with the Euclidean metric. Write
$a_n=7+4n$ for $n=0,1,2,\ldots$, and set
$$
A=\{a_n:n=0,1,2,\ldots\},\qquad B=\{2n:n=1,2,3,\ldots\}.
$$
These sets are closed, and
$$
d(A,B)=1,\qquad A_0=A,\qquad B_0=\{6,8,10,\ldots\}.
$$
Define $T:A\to B$ by
$$
Ta_n=
\begin{cases}
6,&n=0,\\
12,&n=1,\\
2n+2,&n\geq2.
\end{cases}
$$
Since $A$ is discrete, $T$ is continuous. Introduce
$$
q(n)=
\begin{cases}
0,&n=0,\\
1,&n=1,\\
\lfloor n/2\rfloor-1,&n\geq2.
\end{cases}
$$
A direct calculation gives
$$
d(a_m,Ta_n)=1\quad\Longleftrightarrow\quad m=q(n)
\qquad(m,n\geq0).
$$
Thus, the unique proximal point associated with $a_n$ is
$a_{q(n)}$.

For $r,s,t\in\mathbb R$, write
$$
P(r,s,t)=|r-s|+|s-t|+|t-r|
=2\bigl(\max\{r,s,t\}-\min\{r,s,t\}\bigr).
$$
To verify Definition~\ref{d1}, consider indices $i<j<k$.
The corresponding input and proximal perimeters are
$$
P(a_i,a_j,a_k)=8(k-i)
$$
and
$$
P(a_{q(i)},a_{q(j)},a_{q(k)})
=8\bigl(\max\{q(i),q(j),q(k)\}
       -\min\{q(i),q(j),q(k)\}\bigr).
$$

If $i\in\{0,1\}$, we have $k-i\geq2$ and
$$
0\leq q(i),q(j),q(k)
\leq\max\{1,q(k)\}\leq\frac{k-i}{2},
$$
where the last inequality follows from
$q(k)\leq(k-2)/2$ and $k-i\geq2$. The proximal perimeter is
therefore at most one half of the input perimeter.

If $i\geq2$, the function $q$ is nondecreasing on the
relevant indices, and
$$
q(k)-q(i)
=\lfloor k/2\rfloor-\lfloor i/2\rfloor
\leq\left\lceil\frac{k-i}{2}\right\rceil
\leq\frac23(k-i).
$$
For the final inequality, if $k-i$ is even, the ceiling
equals $(k-i)/2$; if $k-i$ is odd, it is at least $3$ and
the ceiling equals $(k-i+1)/2\leq2(k-i)/3$.
Both cases establish
\begin{equation}\label{eqtr}
P(a_{q(i)},a_{q(j)},a_{q(k)})
\leq\frac23 P(a_i,a_j,a_k).
\end{equation}
Since the perimeters are invariant under permutations,
this proves that $T$ is a perimetric proximal contraction
of the first kind with $\alpha=2/3$.

Moreover, $q(n)\leq n$ for every $n$. The equalities
$d(a_i,Ta_j)=d(a_j,Ta_i)=1$ would imply
$$
i=q(j)\leq j,\qquad j=q(i)\leq i,
$$
and hence $i=j$. Therefore $T$ satisfies
Condition~$\Lambda$. We also have $T(A_0)=B_0$.
All the hypotheses of Theorem~\ref{t1} are satisfied.

Finally, $d(a_n,Ta_n)=1$ if and only if $q(n)=n$.
Since $q(0)=0$, $q(1)=1$, and $q(n)<n$ for $n\geq2$,
the best proximity points are exactly $7$ and $11$.
\end{example}

\begin{example}\label{ex3}
Consider the complete metric space $(\mathbb{R}^2, d)$, equipped with the  Euclidean  metric $d$, defined by $d(x,y)=\{|x_1-x_2|^2+|y_1-y_2|^2\}^{\frac{1}{2}}$, $x=(x_1,y_1),y=(x_2,y_2)\in \mathbb{R}^2$. We choose $A=\{(x,1)\in \mathbb{R}^2: -1\leq x \leq 1\}$
and $B=\{(x,0)\in \mathbb{R}^2: -1\leq x \leq 1\}$. Then $d(A,B)=1$, $A_0=A$ and $B_0=B$. We define a mapping $T:A\to B$ by $T(x,y)=\left( \frac{x}{10},0\right)$ for all $(x,y)\in A$. 

Then, for $u=(u_1,1), x=(x_1,1) \in A$, the condition $d(u,Tx)=d(A,B)=1$ yields $u_1=\frac{x_1}{10}$. With this in hand, a simple argument establishes that $T$ is a perimetric proximal contraction of  the second kind, with $\alpha = \frac{2}{3}$. Further, taking $x=(x_1,1), y=(x_2,1)\in A$ with $x\neq y$, the equalities $d(x,Ty)=d(A,B)$ and $d(y,Tx)=d(A,B)$ imply $x_1=\frac{x_2}{10}$ and $x_2=\frac{x_1}{10}$, which is untenable here. Thus, there do not exist distinct points $x, y \in A$ satisfying $d(x,Ty) = d(A,B) = d(y,Tx)$, ensuring that $T$ satisfies Condition $\Lambda$. Moreover, it is evident that $A$ is approximately compact with respect to $B$. In addition, one readily observes that $T(A_0)\subseteq B_0$, and $T$ is injective.  Invoking Theorem \ref{t2}, we conclude that $T$ admits a best proximity point, with a count of at most two; in this case, $T$ has exactly one best proximity point, viz., $(0,1)$.
\end{example}

\begin{example}
Consider the complete metric space $(l_1,d)$ where $l_1$ is the set of all sequences $\{x_n\}$ of real numbers such that $\sum_{n=1}^{\infty}|x_n|<\infty$, and $d(x,y)=\sum_{n=1}^{\infty}|x_n-y_n|$ for $x=\{x_n\}$ and $y=\{y_n\}$. We choose 
\begin{align*}
A&=\Big\{\{x_n\}: 0\leq x_1\leq 2~~\mbox{and}~~x_i=0~~\mbox{for all}~~i\geq 2\Big\}
\end{align*}
and 
\begin{align*}
B&=\Big\{\{x_n\}: 0\leq x_1\leq \tfrac{1}{2},~x_2=1,~~\mbox{and}~~ x_i=0~~\mbox{for all}~~i\geq 3\Big\}.
\end{align*}
Then clearly $d(A,B)=1$. Also, we can calculate
\[
A_0=\Big\{\{x_n\}: x_1\in[0,\tfrac{1}{2}],~x_i=0~~\forall i\geq 2\Big\}\] and                                               \[
B_0=\Big\{\{x_n\}: x_1\in[0,\tfrac{1}{2}],~x_2=1,~x_i=0~~\forall i\geq 3\Big\}.
\]Define a mapping $T:A\to B$ by
\[
Tx=y,~~\mbox{where}~~x=\{x_n\},~~y=\Big\{\tfrac{x_1}{15},1,0,0,\cdots \Big\}.
\]For $u=\{u_n\},~x=\{x_n\}\in A$, the condition $d(u,Tx)=d(A,B)=1$ gives $u_1=\tfrac{x_1}{15}$, and $u_i=x_i=0$ for all $i\geq 2$. From this, a direct calculation establishes that $T$ is a perimetric proximal contraction of the second kind for $\alpha=\tfrac{2}{3}.$  

Further, for $x=\{x_n\}, y=\{y_n\}\in A$ with $x\neq y$, the equalities $d(x,Ty)=d(A,B)$ and $d(y,Tx)=d(A,B)$ yield $x_1=\tfrac{y_1}{15}$ and $y_1=\tfrac{x_1}{15}$, which in turn imply $x_1=y_1=0$, contradicting $x\neq y$. Thus, there do not exist distinct points $x,y\in A$ satisfying $d(x,Ty)=d(A,B)=d(y,Tx)$, ensuring that $T$ satisfies Condition $\Lambda$.  

It is also straightforward to verify that $A$ is approximately compact with respect to $B$. Moreover, one can readily verify that $T(A_0)\subseteq B_0$, while $T$ is injective. Invoking Theorem \ref{t2}, we conclude that $T$ admits a best proximity point, with a count of at most two; in this case, $T$ admits exactly one best proximity point, viz., the zero sequence.
\end{example}


\vskip.5cm\noindent\textbf{Declarations}

\vskip.5cm\noindent\textbf{Competing interests} 

The authors declare no competing interests.

\vskip.5cm\noindent\textbf{Data Availability Statement} 

No datasets were generated or analyzed during the current study.

\vskip.5cm\noindent\textbf{Acknowledgment} 

E. Petrov was supported by the National Research Foundation of Ukraine, project No.~2025.07/0369 ``Qualitative methods of nonlinear analysis of heterogeneous structures''.

\bigskip

CONTACT INFORMATION

\medskip

Hiranmoy Garai \\
Department of Mathematics, Government General Degree College Ranibandh, Bankura, India\\
E-Mail: \url{hiran.garai24@gmail.com}

\medskip

Evgeniy Petrov \\
Institute of Applied Mathematics and Mechanics of the NAS of Ukraine\\
Batiuka str. 19, 84116 Slovyansk, Ukraine\\
E-Mail: \url{eugeniy.petrov@gmail.com} (Corresponding author)

\medskip

Pratikshan Mondal \\
Department of Mathematics, A.B.N. Seal College, Cooch Behar, West Bengal, India\\
E-Mail: \url{real.analysis77@gmail.com} 

\medskip

Lakshmi Kanta Dey \\
Department of Mathematics, National Institute of Technology, Durgapur, India\\
E-Mail: \url{lakshmikdey@yahoo.co.in}

\end{document}